\documentclass[a4paper]{article}
\usepackage{latexsym,amsfonts,amscd,amssymb,enumerate,amstext,amsmath,dsfont}
\usepackage[sc,tiny,center,compact]{titlesec}
\usepackage{amsthm}
\usepackage[all]{xy}
\usepackage{titletoc}
\titlecontents{section}[0pt]{\small\scshape}{\contentslabel{1.5em}}{}{\titlerule*[0.75pc]{.}\contentspage}[\textbullet\ ]

\titlecontents{subsection}[1.5em]{\small\scshape}{\contentslabel{2.3em}}{}{\titlerule*[0.75pc]{.}\contentspage}[\textbullet\ ]
\titlecontents{subsubsection}[3.8em]{\small\scshape}{\contentslabel{3.0em}}{}{\titlerule*[0.75pc]{.}\contentspage}[\textbullet\ ]

\author{M. Grime \\ University of Bristol\\ e-mail: Matt.Grime@bris.ac.uk}

\title{Precovers, localizations and stable homotopy}
\date{August 2007}
\newtheoremstyle{ordinary}{1ex}{0pt}{}{}{\scshape}{.}{\newline}{}
\theoremstyle{ordinary}
\newtheorem{thm}{Theorem}[section]
\newtheorem{defn}[thm]{Definition}
\newtheorem{lem}[thm]{Lemma}
\newtheorem{cor}[thm]{Corollary}
\newtheorem{exmp}[thm]{Example}

\newcommand{\colimit}{\varinjlim}

\newcommand{\catt}{\mathcal{T}}

\newcommand{\catr}{\mathcal{R}}
\newcommand{\cats}{\mathcal{S}}
\newcommand{\cate}{\mathcal{E}}

\newcommand{\catp}{\mathcal{P}}

\renewcommand{\mod}{\mathrm{mod}(kG)}
\newcommand{\Mod}{\mathrm{Mod}(kG)}

\begin{document}
\maketitle

\begin{abstract}
We prove a new localization theorem for stable model categories when the localizing subcategory is generated by a precovering class in the model category. We use this to show how one may explicitly realize certain Bousfield localization functors that arise naturally in the study of relative homological algebra for group algebras.

2000 Mathematics subject classification: 18E10, 18E30,18E35, 18G25, 20J05, 55U35.
\end{abstract}

\section{Introduction}

The results of this paper were originally formulated in order to provide an explicit description of certain Bousfield localization functors in modular representation theory. I am indebted to Peter J{\o}rgensen for correcting an error in the original, and for persuading me that I should rewrite it in the far more general terms of stable homotopy theory and model categories. The original motivation will appear as an application of the main theorem.

Let  $\catt$ be the triangulated quotient of a Frobenius category $\cate$, and suppose that $\catr$  is a precovering class of objects in $\cate$. Consider the localizing subcategory $\langle\catr\rangle^\oplus\subset \catt$. Under the assumption that $\cate$ possesses enough kernels and that  $\catr$ satisfies some mild conditions, one obtains an adjoint to the inclusion of this localizing subcategory, and, moreover, an explicit construction of the adjoint by relating the model structure to the triangulated structure. Although this might seem a little restrictive -- one might wish instead to work solely with a triangulated category without presuming a model structure -- it is frequently the situation that one meets in real life. 

Given $\catt$ and $\catr$ as above we seek a triangle

\[ X_\catr \to X \to X_{\catr^\perp} \to \]
with $X_\catr \in \langle\catr\rangle^\oplus$, and $(\catr,X_{\catr^\perp})_\catt =0$. Let us call such an object the localization triangle of $X$. Let us briefly explain where our assumptions on $\catr$ come from. We  know that  given any object $X$ in $\cate$ there is, by assumption, a precover in $\cate$

\[ R_X \to X\]
and  that in $\cate$ every map from an object in $\catr$ to $X$ factors through $R_X$. Thus if we pass to $\catt$ and complete this map to a triangle 
\[ R_X \to X\to Y \to \]
we would have a plausible candidate for the desired localization triangle in $\catt$. For this to be the triangle we seek, one would need every map  from an object in $\catr$  to factor \emph{uniquely} through the map $R_X\to X$ in $\catt$. We know we may lift to the model category to obtain \emph{a} factorization, but we also know that it will be far from unique. To correct our na\"{\i}ve initial guess, we should make an adjustment to kill these extra maps. This is where we shall assume we have enough kernels. Suppose that the map 

\[R_X \to X\]
has a kernel $K$ in $\cate$ and that $K$ is in $\catr$. Then passing back to $\catt$ one can complete the map $K \to R_X$ to a triangle
\[K\to R_X \to X_\catr \to. \]
Since the composite 
\[K \to R_X \to X\]
is zero, we can complete 
\[  \xymatrix@=15pt{ K\ar[r]\ar[d] & R_X \ar[r]\ar[d]& X_\catr \ar[r] & \\
0\ar[r]&X \ar[r]^{\mathrm{Id}}& X\ar[r] & }\]
to a morphism of triangles, whence we deduce a map 
\[ X_\catr \to  X \]
which we can complete to a triangle
\[ X_\catr\to X \to X_{\catr^\perp}\to. \]
This is the localization triangle we seek. The astute reader will have spotted that we have assumed $K$ is in $\catr$, when it need not be, and in general has no reason to be. However, if $X$ has a finite $\catr$-dimension, then $K$ has dimension one smaller, and this draws the reader to looking for some inductive argument to side-step this issue. We shall present such an argument in theorem \ref{main}. 

We use this construction to show that if $\cats\subseteq\catt$ is a localizing subcategory, and if $\cats=\langle \catr \rangle^\oplus$ for some precovering class, then when the quotient 

\[j\colon \catt \to \catt/\cats\]
exists there is a right adjoint $j_!$..

The application that originally motivated this construction comes from representation theory. Suppose that $R$ is a group algebra, and that $\cate$ is given by some exact structure on $\mathrm{Mod}(R)$ determined by a relative cohomology theory. In general, the finite dimensional objects now compactly generate a proper localizing subcategory, $\cats\subset\catt$, which is obviously the smallest triangulated subcategory containing the pure projective modules. In order to investigate these categories, the author devised this technique to better understand the idempotent modules one obtains from Bousfield localization.

\section{Triangulated categories}
\subsection{Exact categories and resolutions}
Throughout we shall use $\cate$ to denote a Frobenius category which we remind the reader is an exact category with enough projectives and injectives, and for which these two subclasses coincide. We will use $\catp$ for the class of pro/injective objects.  Thus there is a sequence of categories

\[\catp \to \cate \to \frac{\cate}{\catp} =:\catt\]
and $\catt$ is triangulated. We refer the reader to the magnum opus \cite{hovey} and \cite{happel} for readers unfamiliar with exact categories and triangulated quotients.

We shall further suppose that $\cate$ has a contravariantly finite, or precovering class, $\catr$, and that $\cate$ has \emph{enough} kernels. We shall explain what we mean by `enough' momentarily. The reader should use the mnemonic ``we want hypotheses such that $\catr$ can be used to define (left) $\catr$esolutions of objects in $\cate$." To flesh out our hypothesis: to each object $X$ in $\cate$ we can associate (not necessarily uniquely) an object $R_X$, called a precover,  and morphism

\[  R_X\to X\]
with the property that any map from $R\in \catr$ to $X$ factors through $R_X$ in $\cate$. We now explain what the assumption of enough kernels means. We wish to assume that every precover has a kernel. Let $K$ be the kernel of $R_X\to X$. We can now take a precover of $K$, take its kernel, and so on to create an $\catr$-resolution of $X$. The reader should note that we are not assuming that the triple

\[K_{i+1} \to R_i \to K_i\]
is a conflation in $\cate$. In fact if they were all conflations, and $X$ had finite $\catr$-dimension then $X$ would lie in $\langle \catr \rangle^\oplus$, and we would not have to work very hard to deduce what $X_\catr$ is in this case.

It is worth inserting some examples at this point to illustrate the kinds of areas where we wish to apply our ideas.

\begin{exmp} We give two examples using module categories.

\begin{itemize}
\item Let $\cate$ be $\mod$ of a finite group algebra. Pick $H$ some subgroup of $G$ and let $\catr$ be the class of $H$-projective modules. A precover of $X$ is 
\[ \mathrm{Ind}^G_H(\mathrm{Res}_H^G(X)) \to X\]
with the map coming from the counit of the adjunction. Note $\cate$ is abelian so we have enough kernels.

\item Let $\cate$ be $\mathrm{Mod}(A)$ for some frobenius algebra $A$, and let $\catr$ be the class of pure projective modules. If $M$ is written as the direct limit of its finite dimensional submodules 

\[M:= \colimit M_\alpha\]
then a precover is given by the natural map
\[ \coprod_{\alpha \in A} M_\alpha \to M.\]
\end{itemize}
\end{exmp}

We need on more definition. 

\begin{defn}
Let $X$ be in $\cate$, and suppose that 

\[\xymatrix@=15pt{  {\ldots} \ar[r] & R_2 \ar[r] \ar[d] & R_1\ar[r]\ar[d] &  R_0 \ar[r] \ar[d] & X \\
                                            & K_2 \ar[ur]          & K_1\ar[ur]        &     X\ar@{=}[ur]   } \]
is an $\catr$-resolution. If $K_d$ is in $\catr$ for some $d$, then we say $X$ has finite $\catr$-dimension, and if $d$ is minimal with this property we say $X$ has $\catr$-dimension $d$. If no such $d$ exists then $X$ has infinite $\catr$-dimension. We will write $\mathrm{dim}_\catr(X)=d$ to indicate the $\catr$-dimension.
                                            
\end{defn}

\subsection{Finite resolutions}

We continue with the hypotheses of the previous subsection. Let $X$ be any object in $\cate$, and suppose further that the $\catr$-dimension of $X$ is finite.  
We know that there is a diagram with all morphisms and objects considered in $\cate$:

\[ \xymatrix@=15pt{  R_d \ar[r] \ar@{=}[d]& R_{d-1} \ar[r] \ar[d] & &{\cdots} & \ar[r]  & R_{1} \ar[r] \ar[d] & R_0\ar[d] \ar[r]        & X \\
                       K_{d}\ar[ur]                           & K_{d-1}  \ar[ur]                             & & {\cdots} &\ar[ur] &  K_1 \ar[ur]           &   K_{0}\ar@{=}[ur] &        }\]
and in each subdiagram 

\[ K_{i+1} \to R_i \to K_{i} \]
$R_i$ is a  precover of  $K_{i}$, and $K_{i+1}$ is its kernel.

We now pass to the quotient $\catt$ and construct a \emph{homotopy approximation} of the resolution in $\catt$ starting from the left by defining a sequence of maps  $\epsilon_i$, and objects $L_i$ in $\catt$. First, we define $\epsilon_d$ to be the map $R_d\to R_{d-1}$. Now we will construct  $\epsilon_{d-1}$.
By assumption the composite 
\[R_d\to R_{d-1}\to K_{d-1}\]
 is zero,  so there is a diagram 

\[ \xymatrix@=15pt{R_d \ar[r] \ar[d] & R_{d-1} \ar[r] \ar[d] &  L_{d-1} \ar[r]& \\
                    0\ar[r]                 & K_{d-1} \ar@{=}[r]   & K_{d-1} \ar[r]&  }\]
in which the rows are distinguished triangles in $\catt$. We complete to a morphism of triangles
\[ \xymatrix@=15pt{R_d \ar[r] \ar[d] & R_{d-1} \ar[r] \ar[d] & L_{d-1} \ar[d] \ar[r]& \\
                    0\ar[r]                  & K_{d-1} \ar@{=}[r]                        & K_{d-1}  \ar[r] &}\]
                    and define $\epsilon_{d-1}$ to be the composite

\[\xymatrix@=15pt{ L_{d-1} \ar[r] & K_{d-1} \ar[r] & R_{d-2}}.\]
Define $L_{d-2}$ to be the cone of $L_{d-1}\to R_{d-2}$. We note that by construction the composite $L_{d-1}\to R_{d-2}\to K_{d-2}$ is also zero, so we obtain another morphism of triangles

\[ \xymatrix@=15pt{ L_{d-1} \ar[r]^{\epsilon_{d-1}} \ar[d] & R_{d-2} \ar[r] \ar[d] &  L_{d-2}\ar[d] \ar[r]& \\
                    0\ar[r]                  & K_{d-2} \ar@{=}[r]                        & K_{d-2}  \ar[r] &}\]
Now we define $\epsilon_{d-2}$ as the composite $L_{d-2}\to K_{d-2} \to R_{d-3}$. We can iterate this process and we will  end up with a diagram 
  \[ \xymatrix@=15pt{                                      & L_{d-1}\ar[dr]^{\epsilon_{d-1}} \ar@/^1pc/ @{.>}[dd]&  &             &           \ar[dr]& L_1 \ar[dr]^{\epsilon_1}\ar@/^1pc/ @{.>}[dd]  & L_0 \ar[dr]\ar@/^1pc/ @{.>}[dd]& \\
                       R_d \ar[r] \ar@{=}[d] & R_{d-1} \ar[r] \ar[d]  \ar[u]                     & &{\cdots} & \ar[r]  & R_{1} \ar[r] \ar[d] \ar[u] &  R_0\ar[d] \ar[r]\ar[u]  & X \\
                       K_{d}\ar[ur]                           & K_{d-1}  \ar[ur]                             & & {\cdots} &\ar[ur] &  K_1 \ar[ur]                    &   K_{0}\ar@{=}[ur] &        }\]
where all objects and morphisms are in $\catt$. 
                                                             
\noindent
\textsc{Remarks}                                                                                                             
\begin{enumerate}[(i)]
\item $L_0$ is unique, but, \emph{a priori}, not up to unique isomorphism.
\item It lies in the smallest subcategory that contains $R_i$ for $0\leq i \leq d$ and is closed under triangles. 
\item If all of the triples $K_i\to R_{i-1} \to K_{i-1}$ are isomorphic to distinguished triangles, then $L_0$ is stably  isomorphic to $X$.
\item if $X$ is the direct limit of a sequence of objects, $X_i$ in $\cate$, then $L_0$ is the homotopy colimit
\[ \mathrm{hocolim}(X_i)\]
in $\catt$.
\item We genuinely needed the model structure to make this construction: one can define precovers in a triangulated category, however, a triangulated category does not have any non-trivial kernels. The best one could hope for is a diagram
\[ \xymatrix@=15pt{  R_d \ar[r] & R_{d-1} \ar[r] & &{\cdots} & \ar[r]  & R_{1} \ar[r]  & R_0 \ar[r]        & X}\]
in which the composite of two consecutive morphisms was zero. Whilst one could certainly start the construction, one would quickly find that attempts to construct a map from $L_{d-2}$ to $R_{d-3}$ were doomed to failure. 
 \end{enumerate}

\subsection{The main theorem}
We are now in a position to state our first theorem. It asserts that under some mild, but slightly lengthy hypotheses, the object $L_0$ has the unique lifting property.

\begin{thm}\label{main}
Let $\cate$ be a Frobenius category with triangulated quotient $\catt$, and let $\catr$ be a class of objects in $\cate$. Suppose that the following hypotheses are satisfied:

\begin{enumerate}[(i)]
\item $\catr$ is a precovering class in $\cate$;
\item $\catr$  is closed under shifts in $\catt$; 
\item for all $R$ in $\catr$ the injective hull $I(R)$ is in $\catr$.
\end{enumerate}
 Further suppose that $X$ is an object in $\cate$ with $\catr$-dimension $d$ and suppose that  $L_0$ is constructed as above, then there is an isomorphism $(R,L_0)_\catt \cong(R,X)_\catt$ for all $R$ in $\catr$.
\end{thm}
We will break the proof down into a series of straightforward lemmas. First we deal with surjectivity.

\begin{lem}\label{lemone} The map $(R,L_0)_\catt \to (R,X)_\catt$ is surjective.
\end{lem}

\begin{proof}
Let $R\to X$ be any morphism in $\catt$. Lift arbitrarily to a morphism in $\cate$. This must factor through the precover $R_0 \to X$. Recall that the ultimate step in our construction yielded a morphism of triangles
\[ \xymatrix@=15pt{  L_1 \ar[r] \ar[d]& R_0 \ar[r] \ar[d] & L_0 \ar[r] \ar[d]&  \\
                         0 \ar[r] & X \ar@{=}[r] & X \ar[r] & }\]
from which we deduce that the lift $R\to R_0\to X$, when we pass back to $\catt$, factors as $R\to R_0 \to L_0\to X$. 
\end{proof}

Thus it remains to show that the map is an injection.  The proof of injectivity is more complicated. First, we will need a standard observation about how one may choose to factor maps in $\catt$.

\begin{lem}\label{lemtwo} With the hypotheses (i) and (iii) of theorem, \ref{main}. Let $R$ be in $\catr$ and  consider a map $\alpha:R\to X$ in $\cate$. By assumption, this factors  in $\cate$ as 
\[ \xymatrix@=15pt{ R\ar[r]^\beta & R_0\ar[r]^\gamma& X}\]
 since $R_0$ is a precover. Suppose that $\alpha$ is $0$ in the quotient $\catt$, then there is a map in $\delta:R\to R_0$ in $\cate$ such that:  

\begin{enumerate}
\item the composite  $\gamma\delta$ is zero in $\cate$;
\item $\delta = \beta$ in $\catt$.
\end{enumerate}
\end{lem}
\begin{proof}
Since $\alpha=0$ in $\catt$ it factors as $R\to I(R) \to X$. The assumption that $I(R)$ is in $\catr$ is key as this  means that the map  $I(R)\to X$ factors through $R_0$. Set $\delta'$ to be the composite $R\to I(R)\to R_0$. This is the amount by which we need to correct $\beta$, thus we define $\delta:=\beta - \delta'$.
\end{proof}
We prove the next lemma by making one more assumption, and then we shall prove that the assumption is true.
\begin{lem}\label{lemthree}
Suppose that $\alpha$ is in the kernel of the map $(R,L_0)_\catt\to (R,X)_\catt$ and further assume that $\alpha$ factors through $R_0$ in $\catt$, then $\alpha=0$.
\end{lem}
\begin{proof}
This is best illustrated diagrammatically - it is a straightforward argument, but an algebraic proof would overload the reader with notation. We start with a diagram

\[\xymatrix@=15pt{                          &                   &          R\ar[d]^\alpha  & \\
			 L_1\ar[r] \ar[d]& R_0 \ar[r] \ar[d]& L_0 \ar[r]        \ar[d]             & \\
                                    0  \ar[r]   & X \ar@{=}[r] & X                  \ar[r]             & }
                                    \]
where the rows are distinguished triangles. Our extra assumption that $\alpha$ factors through $R_0$ allows us to create a commutative diagram
\[\xymatrix@=15pt{                          &                   &          R\ar[d]^\alpha \ar[dl]& \\
			 L_1\ar[r] \ar[d]& R_0 \ar[r] \ar[d]& L_0 \ar[r]  \ar[d]                   & \\
                                    0    \ar[r]  & X \ar@{=}[r] & X                   \ar[r]            & }\]
an elementary diagram chase shows that the map $R\to R_0\to X$ is zero in $\catt$. Thus we my invoke lemma \ref{lemtwo}, and deduce that we can choose it to be zero in $\cate$. Thus the map $R\to R_0$  must factor in $\cate$, and thus in $\catt$, through  $K_0$. By induction on the $\catr$-dimension (note the case of $\catr$-dimension 0 is trivially true), we can assume that any map from $R$ to $K_0$ factors (uniquely, but that is unimportant)  through $L_1$ in $\catt$. Thus we have a larger commutative diagram

\[\xymatrix@=15pt{ & & R\ar@/_1pc/[ddll] \ar[dd]\ar[dl]\ar[ddl] & \\
                          &           K_0    \ar[d]   &   & \\
			 L_1\ar[r] \ar[d]& R_0 \ar[r] \ar[d]& L_0 \ar[r]  \ar[d]                   & \\
                                    0    \ar[r]  & X \ar@{=}[r] & X                   \ar[r]            & }\]
and we see that $\alpha$ factors through two consecutive maps in a distinguished triangle, and must be $0$ in $\catt$.
\end{proof}

This leaves the extra assumption in lemma \ref{lemthree} to be proven. The proof again inducts on the $\catr$-dimension.

\begin{lem}\label{lemfour}
Any map from $R\to L_0$ factors through $R_0$ in $\catt$. 
\end{lem}
\begin{proof}
The statement is equivalent to the assertion that the natural map 
\[(R,R_0)_\catt \to (R,L_0)_\catt\] is surjective. To see this, we invoke the octahedral axiom to construct

\[\xymatrix@=15pt{ L_1 \ar[r]    \ar@{=}[d]& K_0 \ar[r] \ar[d] & Y  \ar[d]  \ar[r] &  \\
 L_1 \ar[r]^{\epsilon_1} \ar[d]  & R_0 \ar[r] \ar[d]  &   L_0  \ar[r]\ar[d] &  \\
 0\ar[r] & Z \ar@{=}[r] & Z\ar[r] & 0 }\]
where by induction on $\catr$-dimension again we have that $(R,Y)_\catt=0$. Apply the functor $(R,?)_\catt$ to obtain 
\[\xymatrix@=15pt{  & \ar[d]                                            & \ar[d]                                &  \ar[d]                                                      & \ar[d]                                                            &           \\
\ar[r]& (R,L_1)_\catt  \ar@{=}[r]    \ar@{=}[d]    & (R,K_0)_\catt  \ar[r] \ar[d] &0 \ar[d]  \ar[r] & (R,L_1[1])_\catt  \ar@{=}[d] \ar@{=}[r] & \\
\ar[r]&  (R,L_1 )_\catt \ar[r] \ar[d]  & (R,R_0)_\catt \ar[r] \ar[d]  &   (R,L_0)_\catt  \ar[r]\ar@{=}[d] & (R,L_1[1] )_\catt \ar[d] \ar[r] & \\
\ar[r] & 0\ar[r] \ar[d]& (R,Z )_\catt \ar@{=}[r]\ar[d]& (R,Z)_\catt\ar[d]  \ar[r] & 0\ar[d] \ar[r] &  \\
\ar[r]&(R,L_1[1])_\catt \ar@{=}[r]    \ar@{=}[d]    & (R,K_0[1])_\catt  \ar[r] \ar[d] &0 \ar[d]  \ar[r] & (R,L_1[2])_\catt  \ar@{=}[d] \ar@{=}[r] & \\
\ar[r]&  (R,L_0[1])_\catt \ar[r] \ar[d]  & (R,R_0[1])_\catt \ar[r] \ar[d]  &   (R,L_0[1])_\catt  \ar[r]\ar@{=}[d] & (R,L_1[2] )_\catt \ar[d] \ar[r] &  \\ 
 &&&&&}\]
where all rows and columns are exact.  There is a standard diagram chase to be done. The following statements are equivalent.

\begin{enumerate}[(i)]
\item $(R,R_0)_\catt\to (R,L_0)_\catt$ is surjective.
\item $(R,R_0)_\catt \to (R,Z)_\catt$ is surjective.
\item $(R,Z)_\catt\to(R,K_0[1])_\catt$ is zero.
\item $(R,K_0[1])_\catt \to (R,R_0[1])_\catt$ is injective. 
\item $(R,K_0)_\catt\to (R,R_0)_\catt$ is injective.
\end{enumerate}
To prove the last of these conditions, we work in $\cate$. Suppose that $R\to K_0\to R_0$ is zero, then there is a diagram in $\cate$ 

\[\xymatrix@=15pt{ R\ar[r]\ar[d] & I(R)\ar[r]& R[1] \\
                     K_0 \ar[r]& R_0\ar[r] & X }\]
and the top row is a conflation. By assumption there is a map $I(R)\to R_0$ which makes

\[\xymatrix@=15pt{ R\ar[r]\ar[d] & I(R) \ar[d]\ar[r]& R[1] \\
                     K_0 \ar[r]& R_0 \ar[r]& X }\]
commutive. This can be completed to
\[\xymatrix@=15pt{ R\ar[r]\ar[d] & I(R) \ar[d]\ar[r]& R[1]\ar[d] \\
                     K_0 \ar[r]& R_0 \ar[r]& X }\]
since $I(R)\to R[1]$ is a cokernel of $R\to I(R)$. There is one hypothesis yet to be used, which we now invoke. As $\catr$ is closed under shifts, the map from $R[1]\to X$ factors through $R_0$, which implies that the map  $I(R) \to R_0$ factors through $K_0$, hence $R \to K_0$ factors through $I(R)$ as we were required to show.
 \end{proof}

Lemmas \ref{lemone}, \ref{lemtwo}, \ref{lemthree} and \ref{lemfour} complete the proof of the main theorem.

\section{Constructing adjoints}

Suppose that $\cate$, $\catt$, and $\catr$ satisfy the hypotheses of theorem \ref{main}. We can use this to deduce the existence of adjoints to certain inclusion functors via a routine  argument.

\begin{thm}\label{full} Suppose that every object in $\catt$ has finite $\catr$-dimension. Let $\cats$ be the smallest full localizing subcategory of $\catt$ that contains $\catr$, then the inclusion functor 
\[ i: \cats\to \catt\]
has a right adjoint $i^!$.
\end{thm}
 \begin{proof}
 Given $X$ in $\catt$ we may construct $L_0$ as above. Define $i^!(X):=L_0$.  Notice that for $R\in \catr$ theorem \ref{main} gives isomorphisms. 
 
 \[(i(R),X)_\catt \cong (R,L_0)_\catt \cong (R,L_0)_\cats\cong (R,i^!(X))_\cats\]
 Thus we just need to argue that these isomorphisms exist if we replace $R$ with an arbitrary object $S$ in $\cats$. It is clear that the class of objects for which the isomorphisms exist contains $\catr$ and is closed under direct sums and triangles, hence contains $\cats$, and we are done.
\end{proof} 
\begin{cor}
If the quotient  $j:\catt \to \catt/\cats$ exists, then it has a right adjoint, $j_!$.
\end{cor}
\begin{proof}
Define $j_!(X)$ to be the third object in the triangle
\[ L_0\to X \to Y \]
\end{proof}
There is clearly a theorem to be stated if we drop the assumption on finite $\catr$-dimension of all objects.
\begin{thm}\label{partial}
Suppose that $\cate, \catt$, and $\catr$ are as in theorem \ref{main}. Let $\cats$ be the smallest localizing subcategory containing $\catr$. If the inclusion 

\[ i :\cats\to \catt \]
has a left adjoint, $i^!$, and if $X$ has finite $\catr$ dimension, then necessarily $i^!(X)$ is isomorphic to the $L_0$ coming from our construction.
\end{thm}

\section{Applications in relative homological algebra}
We should attempt to convince the reader that the assumptions we have placed on $\cate$ are really the kinds of things she might meet in everyday mathematics. The original motivation for this construction came from modular representation theory.

Let $A$ be a finite dimensional algebra over a field $k$. Suppose that $\catp$ is some class of objects in $\mathrm{Mod}(A)$ such that the quotient

\[\frac{\mathrm{Mod}(A)}{\catp}\]
is triangulated. It is natural to ask how the (isomorphism classes of) finite dimensional objects behave inside this category. Thus we define 

\[ \catr:= \mathrm{Add}(\mathrm{mod}(A))\]
as the class of pure projective objects. This is precovering. The smallest triangulated subcategory containing $\catr$ coincides with the smallest triangulated subcategory containing the finite dimensional objects that is closed under direct sums. Again, we shall use $\cats$ to denote this triangle closure of $\catr$. We can now cite some examples of uses of the localization theorems.

\begin{exmp}
If every module has finite pure projective dimension, then the inclusion  $\cats\to \catt$ has a right adjoint by theorem \ref{full}.
\end{exmp}

\begin{exmp} If the inclusion has an adjoint, then we can calculate it using theorem \ref{partial} for those objects with finite pure projective dimension.
\end{exmp}

Examples of finite pure projective dimension abound, as we now explain. The precise implications depend on one's set theory.
\begin{enumerate}
\item If the cardinality of the underling set of $A$ is $\aleph_t$ for some finite ordinal $t$, then every module has pure projective dimension at most $t+1$. If the cardinality is finite then every module has pure projective dimension 1.
\item For an algebra $A$ over any  field $k$, if the $k$-dimension of a module $M$ is $\aleph_t$, then $M$ has pure projective dimension at most $t+1$.
\end{enumerate}
The reader is referred to \cite{homotopy}[(3.8)] for more examples.  We will finish by fleshing out the bones of these examples. As we have mentioned the original motivation comes from modular representation theory.

Suppose that $G$ is a finite group, $k$ a field of characteristic $p$, and that $p$ divides $|G|$. Let $H$ be a subgroup of $G$. Define $\catp$ to be the class of summands of all modules induced up from $H$. It is now classical that $\Mod/\catp$, the relatively stable category, is triangulated. The triangles correspond to short exact sequences that split on restriction to $H$. One recovers the normal notion of projective, and the usual stable category $\mathrm{StMod}(kG)$, by letting $H$ be the trivial subgroup. 

One can define Rickard modules in $\Mod/\catp$ given a compactly generated subcategory, however the relative case is fundamentally different from the usual stable module category. It is a simple exercise to show that the smallest triangulated subcategory of $\mathrm{StMod}(kG)$ that contains the simple modules and is closed under direct sums is $\mathrm{StMod}(kG)$. It is known, \cite{thesis}[CH.~7 and App.~B] that the finite dimensional objects may (compactly) generate a proper subcategory of the relatively stable category, and thus yield a non-trivial localization functor.  Our construction gives an explicit description of it, under certain conditions: 
\begin{exmp} If $|k|=\aleph_t$, then all modules have finite pure projective dimension, since $|kG|=\aleph_t$.  This is no real restriction, since there are countable algebraically closed fields, and this is more than sufficient for modular representation theory. 
\end{exmp}
\begin{exmp}If $k$ is arbitrary, then one can apply theorem \ref{partial} since the inclusion has an adjoint, as we have observed.
\end{exmp}

 \section{Acknowledgements}
  I would like to thank Peter J{\o}rgensen for commenting on the general case, and in particular for correcting an earlier erroneous proof of lemma \ref{lemfour}, as well as explaining the content of \cite{homotopy} and its implications for the construction presented here.

\begin{small}

\end{small}
\end{document}